\documentclass[a4paper, 10pt, conference]{ieeeconf}      

\IEEEoverridecommandlockouts                              

\usepackage{amsmath} 
\usepackage{amssymb}  

\usepackage{multirow}
\usepackage{cite}
\usepackage{amsfonts}
\usepackage{multirow}
\usepackage{amsmath}
\usepackage{soul,color}
\usepackage[dvipsnames]{xcolor}
\usepackage{stackengine}
\usepackage{amssymb,graphicx}
\usepackage{mathrsfs} 
\usepackage{comment} 

\usepackage{siunitx}
\usepackage{booktabs}

\usepackage{epstopdf}
\newcolumntype{d}[1]{D{.}{.}{#1}}

\usepackage{algorithm}
\usepackage{algorithmicx}  
\usepackage{algpseudocode}
\renewcommand{\algorithmicrequire}{\textbf{Input: }}
\renewcommand{\algorithmicensure}{\textbf{Output: }}

\newcommand{\End}{\hfill $\square$}
\newcommand{\PfEnd}{}

\DeclareMathOperator*{\diag}{diag}

\newcommand{\mbb}[1]{\mathbb #1}
\newcommand{\mbf}[1]{\mathbf #1}
\newcommand{\mcl}[1]{\mathcal #1}
\newcommand{\set}[1]{\mathbb #1}

\newcommand{\X}{\mbb{X}}
\newcommand{\Xf}{\mbb{X}_{\textnormal{f}}}
\newcommand{\Xini}{\mbb{X}_0}
\newcommand{\Xtrain}{\mbb{X}_{\text{train}}}
\newcommand{\Vf}{V_{\textnormal{f}}}
\newcommand{\U}{\mbb{U}}

\newcommand{\I}{\mbb{I}} 
 
\newcommand{\IPL}{\mbb{N}} 
\newcommand{\IP}{\mbb{N}^+} 
\newcommand{\R}{\mbb{R}}

\newcommand{\PF}{ \mathscr{P}(x_k)} 
\newcommand{\PFk}[1]{ \mathscr{P}(x_{#1})} 
\newcommand{\PRz}{ \widehat{\mathscr{P}}(x_k)} 

\newcommand{\PR}{\widehat{\mathscr{P}}(x_k,\tilde{\mathbf{w}}_k)} 
\newcommand{\PRk}[1]{\widehat{\mathscr{P}}(x_{#1},\tilde{\mathbf{w}}_{#1})} 

\newcommand{\OVF}{V} 
\newcommand{\dx}{\textrm{d}x}

\newcommand{\Rnx}{\mbb{R}^{n}}
\newcommand{\Rnu}{\mbb{R}^{m}}

\newtheorem{corollary}{Corollary}
\newtheorem{lemma}{Lemma}
\newtheorem{proposition}{Proposition}
\newtheorem{remark}{Remark}
\newtheorem{example}{Example}

\newtheorem{assumption}{Assumption}

\newcommand{\edit}[1]{\color{black} #1 \color{black} }

\usepackage{etoolbox}
\makeatletter
\patchcmd{\@makecaption}
{\scshape}
{}
{}
{}
\makeatletter
\patchcmd{\@makecaption}
{\\}
{.\ }
{}
{}
\makeatother

\title{NMPC in Active Subspaces: Dimensionality Reduction \\ with Recursive Feasibility Guarantees}

\author{Guanru Pan and Timm Faulwasser$^{\star}$
	\thanks{ $^{\star}$: Corresponding author. G. Pan and T. Faulwasser are with Institute for Energy Systems, Energy Efficiency and Energy Economics, TU Dortmund University, Dortmund,  Germany. Email addresses:
	 guanru.pan@tu-dortmund.de, timm.faulwasser@ieee.org }}

\begin{document}

\maketitle
\thispagestyle{empty}
\pagestyle{empty}

\begin{abstract}
Dimensionality reduction of decision variables is a practical and classic method to reduce the computational burden in linear and Nonlinear Model Predictive Control (NMPC). Available results range from early move-blocking ideas to singular-value decomposition. For schemes more complex than move-blocking it is seemingly not straightforward to guarantee recursive feasibility of the receding-horizon optimization. Decomposing the space of decision variables related to the inputs into active and inactive complements, this paper proposes a general framework for effective feasibility-preserving dimensionality reduction in NMPC. We show how---independently of the actual choice of the subspaces---recursive feasibility can be established. Moreover, we propose the use of global sensitivity analysis to construct the active subspace in data-driven fashion based on user-defined criteria. 
Numerical examples illustrate the efficacy of the proposed scheme. Specifically, for a chemical reactor we obtain a significant reduction by factor $20-40$  at a closed-loop performance decay of less than $0.05\%$.
\end{abstract}


\section{Introduction}
The last two decades have seen tremendous research progress on linear and Nonlinear Model Predictive Control (NMPC) in terms of system theoretic analysis---which covers stochastic, robust, distributed, as well as economic formulations, cf. \cite{Faulwasser18c,Gruene17a,Rawlings20,Coulson19}---and  in terms of numerical methods. The latter range from multi-parametric programming for pre-computation of  explicit feedback laws~\cite{Bemporad02a,Tondel03} to tailored online solutions~\cite{Houska11a,Kapernick14a}. 

It is well-known that  NMPC can rely on suboptimal solutions~\cite{Allan17,Graichen10a,Scokaert99} which, e.g.,  arise from early truncation of the iterative numerical schemes or from the solution of the underlying Optimal Control Problem (OCP) with a low-dimensional input parametrization, i.e., solving the  OCP  in a low-dimensional input subspace. 
In the context of continuous-time OCPs the construction and exploitation of low-dimensional (\textit{parsimonious}) input parametrizations is a classic theme, cf.~\cite{Teo91} and the introductory overview in~\cite{Aydin18a}. In case of discrete-time MPC three established approaches towards structured computation of input subspaces can be identified:  move blocking \cite{Cagienard07,Shekhar15}, Singular Value Decomposition (SVD)~\cite{Bemporad20,Hara12,Ong14,Rojas04,Unger12}, and Laguerre functions \cite{Khan13,Lawrynczuk20,Wang04}.

Move blocking has already been mentioned in the early textbook~\cite{Maciejowski02}. The idea is to limit the number of control decision variables by keeping the inputs constant during pre-specified periods. Adapting the blocking structure at each MPC iteration, \cite{Cagienard07} presents stability and feasibility guarantees for linear MPC. 
In \cite{Shekhar15} stability and feasibility guarantees are given for fixed blocking structures in linear MPC; the authors propose an approach to obtain a blocking structure, which implies a maximized region of attraction of the closed loop.

Alternatively, in \cite{Unger12} SVD is performed on a matrix consisting of optimal input sequences. This parametrization approach is also suitable for nonlinear MPC. However, \cite{Unger12} does not provide NMPC stability guarantees for the low-dimensional setting. In \cite{Hara12} the low-dimensional subspace is characterized by applying SVD to the Hessian of the condensed OCP, while \cite{Ong14} employs SVD on a matrix constructed by condensing equality constraints induced by the linear dynamics to determine a low-dimensional subspace.
Recently,  \cite{Bemporad20} proposes a scheme that combines SVD and K-means clustering.
 The contributions \cite{Bemporad20,Hara12,Ong14} provide feasibility guarantees for linear MPC by considering a suboptimal feasible control  sequence as a fixed or scaled parametric part of the low-dimensional solution.
Moreover, Laguerre polynomials and other orthogonal bases can be used for low-dimensional parametrization of the control sequence \cite{Khan13,Wang04}.
The latter two approaches have feasibility and stability guarantees for linear MPC. Moreover, \cite{Lawrynczuk20} proposes an NMPC scheme with reasonable performance based on online trajectory linearization and parametrization with Laguerre functions.

A commonality of all works mentioned above is that the underlying constructions of low-dimensional input subspaces are restricted to linear MPC or linearized versions of NMPC. To the best of the authors' knowledge, results on low-dimensional NMPC with recursive feasibility and stability guarantees are so far not available in the literature. 
Extending the preliminary results of the master thesis~\cite{kit:pan20}, the present paper addresses this gap. 
Our first contribution is a generalized framework for NMPC in low-dimensional subspaces with feasibility and stability guarantees. Importantly, our approach does not depend on the specific choice or the considered construction of the subspace. 
Similarly to \cite{Bemporad20,Hara12,Ong14,Shekhar15}, we establish recursive feasibility by constructing a guess from the last known optimal input sequence. In contrast to \cite{Bemporad20,Hara12,Ong14,Shekhar15} we allow for nonlinear dynamics.  Moreover, we show that our approach can be extended beyond the stabilizing setting provided a constructor for feasible guesses in the original high-dimensional input space is known.  
 
 The second contribution is a data-driven framework for the construction of low-dimensional input subspaces for NMPC. To this end, we leverage a method from global sensitivity analysis known as \textit{active subspaces}~\cite{Constantine15}.  The computation of active subspaces based on global sensitivity analysis has been considered for measurement-based real-time optimization under uncertainty~\cite{Costello16,Singhal18}.  However, the concepts of active subspaces have not been used for MPC yet.
  Considering a user-defined function, which maps the initial condition onto a specific sensitivity criterion with respect to the input decisions,  we suggest approximating the corresponding covariance matrix over a set of initial conditions via sampling. In turn, this enables to compute active subspaces relevant for NMPC via SVD. 
 For a specific linear-quadratic example, we observe that the proposed MPC scheme achieves minuscule closed-loop performance loss in the range of $10^{-2} \%$ with only 6 decision variables compared to 120 in the full-order problem.
 We show for a nonlinear robot example that  dimensionality reduction of the input decision variables by a factor $10$  leads to a decay of  closed-performance of stabilizing NMPC averaged over several initial conditions of only $2.5\%$. Moreover, for economic NMPC of a chemical reactor exhibiting the turnpike phenomenon, we observe that a reduction by a factor $20-40$ leads to a closed-loop performance decay of less than $0.05\%$. 
 	We observe in simulations (cf. Section~\ref{Sec:results}) that nominal NMPC schemes with the same number of decision variables (via shorter horizons) as used in the reduced subspaces  either encounter infeasibility (if terminal constraints are present) or result in significant closed-loop performance loss. 

The reminder of the paper is structured as follows: Section~\ref{sec:preliminaries} gives details about the considered setting; Section~\ref{Sec:activesubspace} and Section~\ref{sec:sensi} present the main results. Section~\ref{Sec:results} considers  several numerical examples; the paper ends with conclusions in Section~\ref{sec:con}.

\section{Preliminaries}\label{sec:preliminaries}
We consider nonlinear discrete-time systems 
\begin{equation}\label{eq:dynamic}
x^+ = f(x,u),\quad x(0) = x_0.
\end{equation}
where $f: \mathbb{X} \times \mathbb{U} \to \mathbb{X}$ is the continuous state transition map,  $x \in \mathbb{X} \subseteq \Rnx$  is the state, and  $u \in \mathbb{U} \subseteq \Rnu$ is the input. We assume that the origin $(x,u)=(0,0)$ is an equilibrium of \eqref{eq:dynamic}.
The constraint sets $\mathbb{X}$ and $\mathbb{U}$ are closed and contain the origin in their interior. 
As usual in NMPC, at time instant $k \in \IPL$, we consider the following OCP  with horizon $N\in \IP$,
\begin{subequations}\label{eq:OCPk}
\begin{align}
V(x_k) \doteq \min\limits_{\mbf{u}_k,\mbf{x}_k}\sum_{i=0}^{N-1}\ell(x_{k+i|k},&u_{k+i|k})+\Vf(x_{k+N|k}) , \\
\text{s.t.}\quad   \forall i \in \set{I}_{[0,N-1]}, \quad x_{k|k}&=x_k \\
    x_{k+i+1|k}&=f(x_{k+i|k}, u_{k+i|k})\ \ \\
    u_{k+i|k}&\in  \mathbb{U} ,\,  x_{k+i|k}\in  \mathbb{X} \label{eq:constraints}\\
    x_{k+N|k}& \in \Xf,\label{eq:tercon}
\end{align}
\end{subequations}
where $\mathbf{u}_k \doteq [u_{k|k}^\top, ..., u_{k+N-1|k}^\top]^\top$ denotes the predicted control sequence at time instant $k$ and $\mathbf{x}_k \doteq [x_{k|k}^\top,..., x_{k+N|k}^\top]^\top$ the corresponding predicted state sequence. The stage cost $\ell: \Rnx \times \Rnu \to \mbb{R}$ and the terminal cost $\Vf: \Rnx \to \mbb{R}$ are continuous. The set of feasible initial conditions of \eqref{eq:OCPk} is denoted as $\hat{\mbb{X}}_0 \subseteq \mathbb{X}$. Henceforth, we consider $\Xini \subseteq \hat{\mbb{X}}_0 \subseteq \mathbb{X}$, i.e., a compact set of initial conditions of interest and which contains the origin in its interior. Whenever necessary, optimal solutions are indicated by superscript $\cdot^\star$.

The following assumptions  are frequently used to ensure recursive feasibility  and stability via terminal constraints, cf.~\cite{Chen98,Mayne00}. 
\begin{assumption}[Terminal ingredients]\label{ass:cond}~\\ \vspace*{-0.3cm}\hspace{5cm}
\begin{itemize}
\item[A1] There exists a function $\alpha\in\mathcal{K}_{\infty}$ such that $\ell(x,u) \geq \alpha(\| x \|)$, $\forall(x,u)\in \mathbb{X} \times \mathbb{U}$ and $\ell(0,0)=0$.
\item[A2]  There exists a feedback $\kappa : \Xf \to \mathbb{U}$ such that for all $x \in \Xf$, $f(x,\kappa(x)) \in \Xf$, $\Vf(0)=0$, and 
\begin{equation}
\Vf\left( f\left(x, \kappa(x)\right)\right) -\Vf(x) \leq -\ell(x,\kappa(x)).\label{ass:FF}
\end{equation}
\item[A3] There exists a function $\beta\in\mathcal{K}_{\infty}$ such that $\OVF(x) \leq \beta(\|x\|)$, $\forall x \in \Xini$. \End
\end{itemize}
\end{assumption}
The above conditions use the usual notion of class $\mcl K_\infty$ functions.\footnote{A function $\alpha: \mathbb R^+_0 \to \mathbb R^+_0$ is said to belong to class $\mcl K$ if its is strictly monotonously increasing and $\alpha(0) = 0$. It belongs to class $\mcl K_\infty$ if additionally $\lim_{s\to \infty} \alpha(s) = \infty$.  } 

Observe that, formally, the decision variables of OCP~\eqref{eq:OCPk} are the predicted inputs $\mbf{u}_k$ \textit{and} the predicted states $\mbf{x}_k$. However, the predicted states $x_{k+i|k}$, $i \in \I_{[1,N]}$, depend on $x_k$ and on the choice of $\mbf{u}_k$. They can be expressed recursively as
\begin{equation}\label{eq:condensing}
x_{k+i|k}= f\left(f\left(f\left(x_k,u_{k|k}\right),u_{k+1|k}\right),...,u_{k+i-1|k}\right).
\end{equation}
After the elimination of state variables, which in the context of linear quadratic MPC is also known as \textit{condensing}, one can  reformulate \eqref{eq:OCPk} as follows 
\begin{subequations}
\label{eq:fullOCP}
\begin{align}
\mathscr{P}(x_k):\ \ \  \min_{\mathbf{u}_k} J(x_k,\mathbf{u}_k)\label{eq:condensedObj}\\
\text{s.t. }
\label{eq:condensedcon}
g(x_k ,\mathbf{u}_k) \leq 0,
\end{align}
\end{subequations}
where $J: \mathbb{X} \times \mathbb{U}^N \to \mathbb{R}$ and $g: \mathbb{X} \times \mathbb{U}^N \to \mathbb{R}^{n_g}$ are the rewritten objective and constraints, respectively. Notice that \eqref{eq:condensedcon} summarizes \eqref{eq:constraints}-\eqref{eq:tercon}. For the sake of simplified exposition, our subsequent discussions focus on the condensed problem~\eqref{eq:fullOCP}.

\section{Exploiting Active Subspaces in NMPC} \label{Sec:activesubspace}
Consider 
\begin{equation}\label{eq:InputProjection}
	T \hat{\mathbf{u}} = \begin{bmatrix}
		T_1 & T_2
	\end{bmatrix}\begin{bmatrix}
		\mathbf{v} \\ \mathbf{w}
	\end{bmatrix}
	= \mathbf{u},
\end{equation}
that is, the orthonormal matrix $T\in \mathbb{R}^{Nm \times Nm}$ with $TT^\top = I$  which maps $\mathbf{u} \in \mathbb{R}^{Nm}$ to new coordinates $(\mathbf{v},\mathbf{w})$.
Specifically, due to the orthonormality of $T$, we have $T^{-1} = T^\top$ and thus
 \begin{equation}
 	\begin{bmatrix}\label{eq:InputProjection_inverse}
 		\mathbf{v} \\ \mathbf{w}
 	\end{bmatrix} =T^\top  \mathbf{u} = \begin{bmatrix}
 		T_1^\top \mathbf{u}  \\ T_2^\top \mathbf{u}
 	\end{bmatrix}.
 	\end{equation}
  For the sake of simplicity and without loss of generality, we consider the case that $T_1 \in \mathbb{R}^{Nm \times q}$ and $T_2 \in \mathbb{R}^{Nm \times Nm-q}$ are identical to the first $q$ and the remaining $Nm-q$ columns of $T$, respectively.|| 
Following the terminology of~\cite{Constantine15}, we refer to the space spanned by $T_1$ as \textit{active subspace} $\mathcal{AS}=\mathrm{col}(T_1)$. If $q\ll Nm$, $\mathcal{AS}$ is a low dimensional subspace of  $\mathbb{R}^{Nm}$. The complement spanned by $T_2$ is denoted as \textit{inactive subspace} $\mathcal{IS}= \mathrm{col}(T_2)$. 

At this point, we postpone answering the question of how to compute the projectors $T_1$ and $T_2$ to Section~\ref{sec:sensi}. However, temporarily supposing that $T_1$ and $T_2$ are chosen such that $q \ll Nm$, i.e. $\dim v \ll \dim v + \dim w$, we are interested in analyzing the condensed problem \eqref{eq:fullOCP} if the optimization is performed only with respect to $\mathbf{v}_k=T_1^\top\mathbf{u}_k$. That is, we neglect, respectively, we restrict the inactive component $\mathbf{w}_k=T_2^\top\mathbf{u}_k$.
For starters, setting $\mathbf{w}_k=0$, we  obtain the reduced problem 
\begin{subequations}\label{eq:naivePR}
\begin{align}
 \PRz :\quad  \mathop{\min}_{ \mathbf{v}_k}J(x_k,T_1\mathbf{v}_k ) \\
      \text{s.t.    }
      g(x_k,T_1\mathbf{v}_k) \leq 0. 
\end{align}
\end{subequations}
However, except for specific matrices $T$---e.g., obtained via move-blocking---it is not clear how to enforce recursive feasibility of the NMPC loop.
 The main reason is that in the $(\mathbf{v},\mathbf{w})$ coordinates the multi-stage structure is usually lost, i.e., the problem in the reduced subspace is more densely coupled and the block structure of Jacobians and Hessians is lost. 
The next example shows that the optimal control problem can be infeasible if the inactive subspace component $\mathbf{w}$ is neglected.
\begin{example}[Infeasibility of $ \PRz$] Consider an OCP for a linear controllable system $x^+ = A x +B u$ with $x\in \mathbb{R}^2$ and $u\in \mathbb{R}^1$.  The stage cost is $\ell(x,u) = x^\top x +  u^2$. The horizon is $N=5$. 
 There are no state or input constraints and only the terminal constraint $x_5 = 0 \in \R^2$ is considered. 
The projection $T_1= [1,0,0,0,0]^\top$ implies that we consider the first input, $u_1$, as $\mathbf{v}$ and the other decision variables are forced to be zero. However, the  two-dimensional terminal equality constraint  $x_5 = 0 \in \R^2$ requires at least two degrees of freedom. Hence, in this case the reduced-order OCP $\widehat{\mathscr{P}}(x_0)$ will likely be infeasible for most $x_0\in\mathbb{R}^2$ while the full-order OCP $\mathscr{P}(x_0)$ is feasible  for all $x_0\in\mathbb{R}^2$. \End
\end{example}

To enforce feasibility of the reduced-order OCP, we consider the vector $\mbf{w}_k$ which is passed as an additional parameter to $\mathscr{P}(x_k)$. Suppose that a structured procedure to construct feasible solutions to $\mathscr{P}(x_{k})$ from an optimal solution to $\mathscr{P}(x_{k-1})$ is known---which is the case if Assumption~\ref{ass:cond} holds. 
In other words,  we consider a feasible suboptimal solution $\tilde{\mathbf{u}}_k$ of $\mathscr{P}(x_{k})$ modeled as
\begin{equation}\label{eq:tildeuk}
	\tilde{\mathbf{u}}_k = T_1\tilde{\mathbf{v}}_k + \mu_k T_2\tilde{\mathbf{w}}_k,
\end{equation}
where $\mu_k $ is temporarily fixed to the value $\mu_k=1$. Then,
the problem
\begin{subequations} \label{eq:PR}
\begin{align}
\PR:\,  \mathop{\min}_{ \mathbf{v}_k,\mu_k}J(x_k,T_1\mathbf{v}_k+ \mu_k T_2  \tilde{\mathbf{w}}_k ) \\
      \text{s.t.    }  g(x_k,T_1\mathbf{v}_k+ \mu_k T_2  \tilde{\mathbf{w}}_k) \leq 0.
\end{align}
\end{subequations}
is of the same dimension as~\eqref{eq:naivePR}. Apparently, it
admits the feasible solution $\mathbf{v}_k=\tilde{\mathbf{v}}_k,\, \mu_k = 1$, cf. \eqref{eq:tildeuk}. 
To fully exploit this observation, one needs to obtain $\tilde{\mathbf{w}}_k$ at run-time of the NMPC controller---yet prior to the optimization. Furthermore, in \eqref{eq:PR} we introduce $\mu_k \in \mathbb{R}$ as an extra scalar decision variable. Compared to the fixed value $\mu_k =1$, this consideration improves performance of the closed-loop since the suboptimal guess $\tilde{\mbf{w}}_k$ can be scaled to zero.

In case of stabilizing NMPC, the computation of $\tilde{\mbf{w}}_k$ follows the classic construction of feasible solutions~\cite{Chen98,Rawlings20}:
assuming the last predicted input trajectory $\mathbf{u}_{k-1}$ is known, one obtains
\begin{equation}\label{eq:recurfeaupdate}
\tilde{\mathbf{u}}_{k} = [u_{k|k-1}^\top,...,u_{k+N-2|k-1}^\top,\kappa(x_{k+N-1|k-1})^\top]^\top,
\end{equation} 
 i.e. by shifting $\mathbf{u}_{k-1}$ and appending $\kappa(x_{k+N-1|k-1})$, cf. Assumption~\ref{ass:cond}. Moreover, as outlined above, with $\tilde{\mbf{w}}_k = T_2^\top \tilde{\mbf{u}}_k$, cf. \eqref{eq:InputProjection_inverse}, we ensure feasibility of $\PR$.

\subsection*{Stability and Feasibility of Active-subspace NMPC}\label{sec:RMPC}
Next, we consider the active-subspace NMPC scheme based on $\PR$ from \eqref{eq:PR} as summarized in Algorithm \ref{alg:reducedMPC}. Given the current state estimate/measurement $x_k$ and the feasible guess $\tilde{\mathbf{u}}_k$, the on-line optimization phase solves  the reduced-order problem $\PR$ for $(\mbf{v}^\star_k,\mu^\star_k)$ with $\tilde{\mathbf{w}}_k = T_2^\top \tilde{\mathbf{u}}_k$. If
the employed optimization routine satisfies
\begin{equation}\label{eq:CostDecrease_star}
J(x_k,T_1{\mathbf{v}}^\star_k + \mu^\star_k T_2\tilde{\mathbf{w}}_k) \leq J(x_k,\tilde{\mathbf{u}}_k),
\end{equation}
we reconstruct the control sequence $\mathbf{u}_k = T_1 \mathbf{v}_k^\star+ \mu_k^\star T_2 \tilde{\mathbf{w}}_k$; otherwise, we  set $\mathbf{u}_k  = \tilde{\mathbf{u}}_k$. This way, we ensure the cost decrease condition 
\begin{equation}\label{eq:CostDecrease}
	J(x_k, \mathbf{u}_k) \leq J(x_k,\tilde{\mathbf{u}}_k),
\end{equation}
which is helpful to show stability \cite{Scokaert99}.

 Observe that in Algorithm \ref{alg:reducedMPC} the full-order OCP $\PFk{0}$ is solved for the feasible guess $\tilde{\mbf{u}}_0$ at $k =0$ to initialize $\tilde{\mbf{w}}_0$, and then at each sampling instant $k\in \IP$, we update $\tilde{\mathbf{w}}_k$  with the last predicted inputs $\mathbf{u}_{k-1}$ by \eqref{eq:recurfeaupdate}. In case an alternative construction of feasible initial guesses for $\PR$ is available, one may straightforwardly modify this step. We note that due to the construction of the feasible guess, the NMPC scheme becomes a dynamic feedback strategy as $\tilde{\mathbf{w}}_k$ can be regarded as a memory state of the controller.

\begin{algorithm}[t]
        \caption{Active-subspace NMPC with $\PR$ }\label{alg:reducedMPC}
\algorithmicrequire $T_1$, $T_2$, $\kappa$, $x_0\in \Xini$
\begin{algorithmic}[1]
\State Solve $\PFk{0}$  for $\tilde{\mbf{u}}_0$, set $\tilde{\mbf{w}}_0 \leftarrow T_2^\top \tilde{\mbf{u}}_0$, $k \leftarrow 0$
\State Estimate $x_k$, solve $\PR$ from \eqref{eq:PR} for $(\mathbf{v}_k^\star,\mu_k^\star)$ 
\label{step:2}
\If{Condition \eqref{eq:CostDecrease_star} holds} 
\State  $\mathbf{u}_k \leftarrow T_1 \mathbf{v}_k^\star+ \mu_k^\star T_2 \tilde{\mathbf{w}}_k$ 
\Else \State  $\mathbf{u}_k \leftarrow \tilde{\mathbf{u}}_k$ \EndIf
\State Apply the first step of $\mathbf{u}_k$
\State \label{step:fesupdate} Update $\tilde{\mathbf{u}}_{k}$ by \eqref{eq:recurfeaupdate}, $\tilde{\mbf{w}}_k \leftarrow T_2^\top \tilde{\mathbf{u}}_{k}$ 
\State $k \leftarrow k+1$, and go to Step~\ref{step:2}
\end{algorithmic}
\end{algorithm}

The following propositions establish recursive feasibility and stability of NMPC based on $\PR$ from \eqref{eq:PR}, respectively, based on Algorithm \ref{alg:reducedMPC}.
\begin{proposition}[Recursive feasibility in  subspaces] \label{prop:recursive_feasibility}
	 Let Assumption \ref{ass:cond} hold and consider the active-subspace NMPC based on \eqref{eq:PR} and Algorithm \ref{alg:reducedMPC}. If $\PRk{k}$ is feasible for $k \in\IPL$, i.e., there is a $( \mathbf{v}_{k}, \mu_{k})$ such that $\mathbf{u}_{k } = T_1 \mathbf{v}_{k}+ \mu_{k} T_2 \tilde{\mathbf{w}}_{k}$ satisfies the constraints of \eqref{eq:PR}, then $\PRk{k+1}$ is also feasible. 
\End
\end{proposition}
\begin{proof}
Suppose $\PRk{k}$ is feasible, then there exits a solution $( \mathbf{v}_{k}, \mu_{k})$ that allows to recover $\mathbf{u}_{k } = T_1 \mathbf{v}_{k}+ \mu_{k} T_2 \tilde{\mathbf{w}}_{k}$.  Due to feasibility,  the corresponding predicted state trajectory satisfies  $x_{k+N|k} \in \Xf$. As we consider the nominal setting, i.e., no plant-model mismatch, we have $x_{k+1} = x_{k+1|k}$ at time instant $k+1$. Now we construct $\tilde{\mathbf{u}}_{k+1}$ by shifting $\mathbf{u}_{k}$ as shown in \eqref{eq:recurfeaupdate}. Then, due to  Assumption \ref{ass:cond}, the state sequence corresponding to this guess satisfy $f\big(x_{k+N|k},\kappa(x_{k+N|k})\big) \in \Xf$ and  $\kappa(x_{k+N|k}) \in \U$. Therefore, $\tilde{\mathbf{u}}_{k+1}$ constructed as in \eqref{eq:recurfeaupdate} is feasible in $\mathscr{P}(x_{k+1})$. Hence $\PRk{k+1}$ with $\tilde{\mbf{w}}_{k+1} = T_2^\top \tilde{\mbf{u}}_{k+1}$, is feasible as it admits a feasible solution $(\mathbf{v}_{k+1} = T_1^\top \tilde{\mbf{u}}_{k+1}, \mu_{k+1} =1)$. \PfEnd
\end{proof}
\begin{proposition}[Stability via active subspaces]
Consider the active-subspace NMPC based on \eqref{eq:PR} and Algorithm \ref{alg:reducedMPC} for any orthogonal $T = \begin{bmatrix}
T_1 & T_2
\end{bmatrix}$. Let Assumption \ref{ass:cond} hold and suppose that $\PF$ is feasible for $x(k=0) =x_0$.
Then, the origin $x = 0$ is an asymptotically stable equilibrium of the closed loop and the region of attraction is given by the set of initial conditions $\Xini$ for which $\PF$ is feasible. \End
\end{proposition}
\begin{proof}
In NMPC with terminal constraints, suboptimal solutions lead to stability if the  cost-decrease \eqref{eq:CostDecrease} holds, cf.  \cite[Theorem 14]{Allan17} and \cite{Scokaert99}.
Hence, the assertion follows from the observation that, for any active subspace based on an orthogonal matrix $T = \begin{bmatrix}
T_1 & T_2 \end{bmatrix}$, Algorithm \ref{alg:reducedMPC} provides feasible but suboptimal solutions to \eqref{eq:fullOCP} which also satisfy \eqref{eq:CostDecrease}.
\PfEnd
\end{proof}
Note that in checking the decrease condition \eqref{eq:CostDecrease} one may exploit the time coupling. That is, for the suboptimal feasible case, one can compute the right-hand-side in \eqref{eq:CostDecrease_star} and \eqref{eq:CostDecrease} from the last predicted solution and the evaluation of the one-step cost resulting from the terminal feedback law. An alternative to checking the decrease conditions explicitly is to rely on line-search methods in the numerical solution to the reduced-order OCP which enforces the cost decrease during the iterations. 

It is worth to be remarked that the core idea to establish recursive feasibility is  the projection of  a feasible $\tilde{\mathbf{u}}_k$ onto the inactive subspace. It does not rely on the specific construction of the feasible guess~\eqref{eq:recurfeaupdate} nor on the choice of the subspaces. Let $F: \mathbb{U}^N \to  \mathbb{U}^N$ be a generic  constructor for feasible guesses
\begin{equation}\label{eq:FeasContructor}
 \tilde{\mathbf{u}}_{k+1}= F(\mathbf{u}^\star_k),
\end{equation}
i.e., it maps the predicted optimal control sequence at time step $k$ to feasible guess at time step $k+1$. Examples of such constructions which differ from the usual \textit{appending the terminal feedback at the end} in \eqref{eq:recurfeaupdate} are, e.g., considered in the context of economic MPC~ \cite[Section 4.2]{Faulwasser18c}. Specifically, in \cite{epfl:faulwasser15g} an exact construction based on turnpike properties of OCPs has been proposed. A simple heuristic approximation for economic MPC is to  split $\mathbf{u}_k$ into two pieces defined on $\set{I}_{[m+1, Lm-1]}$ and $\set{I}_{[Lm, Nm] }$ with $L \in \I_{[1,N]}$ and to add the optimal steady-state input in the middle of the horizon, cf.~\cite{Gruene13a}.
In other words, the recursive feasibility result of Propostion~\ref{prop:recursive_feasibility} can be extended to generic feasibility constructions in NMPC.
\begin{corollary}[Feasibility w generic constructors] Consider active-subspace NMPC based on \eqref{eq:PR} and Algorithm \ref{alg:reducedMPC}, wherein \eqref{eq:recurfeaupdate} is replaced with \eqref{eq:FeasContructor} in Step~\ref{step:fesupdate}. Then, the active-subspace NMPC based on \eqref{eq:PR} is recursively feasible.  \End
\end{corollary}

\begin{remark}[Stability w/o terminal constraints]
In case one wishes to neglect the terminal constraints~\eqref{eq:tercon}, one has different options: (i)  replace the explicit terminal constraint set $\Xf$ by a scaled variant of the terminal penalty, i.e., $\beta\Vf$ with $\beta \gg 1$, see~\cite{Limon06a}. (ii) alternatively, one could remove the terminal constraint set $\Xf$ and increase the prediction horizon, cf.~\cite{tudo:pan21a}. For both approaches the active subspace stability result trivially holds, as the terminal constraint~\eqref{eq:tercon} is still implicitly satisfied. The question of how to guarantee stability of active subspace NMPC without any implicit or explicit consideration of the terminal constraint $\Xf$  could arguably be approached by cost-controllability conditions on the active subspace~\cite{Gruene09a}. A detailed analysis is, however, beyond the scope of this note.  \End
\end{remark}

\section{Computation of Active Subspaces via Global Sensitivities}\label{sec:sensi}
So far we have shown that active-subspace NMPC ensures asymptotic stability independently of the specific choice of the active subspace. It remains to propose a constructive approach to obtain the projectors $T_1$ and $T_2$ which specify the active and the inactive subspaces. To this end, we use global sensitivity analysis as a tool to compute $T_1$ and $T_2$~\cite{Constantine15}.

Let a generic sensitivity function $S: \Xini \to \mathbb{R}^{Nm}$ related to the original optimization problems~\eqref{eq:OCPk}, respectively, \eqref{eq:fullOCP} be given. We consider 
\begin{equation}
\label{eq:pd}
C = \frac{1}{| \Xini |}\int_{\Xini} S(x) S(x)^\top \dx,
\end{equation}
i.e.,  $C$ is the covariance matrix of $S$ over the consider set of initial conditions $\Xini$ and $|\cdot|$ denotes the Lebesgue measure of a set. We note that $C$ is positive semi-definite and symmetric. Thus $C$ admits the non-negative real eigendecomposition,
\begin{equation} \label{eq:eigen_decomposition}
\begin{aligned}
C &= T \Sigma T^\top, \\ 
\Sigma &= \text{diag}(\sigma_1,..,\sigma_{Nm}),\ \sigma_1\geq ...\geq\sigma_{Nm}\geq 0,
\end{aligned} 
\end{equation}
where $T$ is the $Nm\times Nm $ orthonormal matrix spanned by the normalized eigenvectors of $C$ and $\sigma_1\geq \dots\geq\sigma_{Nm}$ are the corresponding eigenvalues. The dimensionality reduction introduced in  Section~\ref{Sec:activesubspace}  means to consider the first $q$ eigenvectors which span the active subspace $\mathcal{AS}$. Naturally, the choice of the dimension $q \in \mathbb{I}_{[1,Nm]}$ is problem-specific and relies on spectrum of $C$. Indeed, if all eigenvalues of $C$ would be identical---which we have not observed in any example---then it might be advisable to consider a different sensitivity function $S$.

The next result is inspired by \cite{Constantine15} and motivates the use of the label \textit{active subspace}. It shows that if the dimension of the active subspace $q$ is chosen such that  $\sum_{i=1}^q \sigma_i \geq \sum_{i=q+1}^{Nm}\sigma_i$, then in average the performance varies much more in $\mathcal{AS}$ than in $\mathcal{IS}$.
Specifically, for $S(x) = \nabla_\mathbf{u} J(x,\mathbf{u}^\star(x))$ the following result holds.

\begin{lemma}
\label{prop:22}
 The mean-squared gradients of $J(x,\mathbf{u}^\star)$ with respect to $\mathbf{v}$ and $\mathbf{w}$ defined in \eqref{eq:InputProjection} satisfy,
 \begin{align*}
    \frac{1}{|\Xini|}\int_{\Xini}\|\nabla_\mathbf{j} J(x,\mathbf{u}^\star(x))\|^2 \dx &=\sum_{i \in \mathcal{I}_\mathbf{j}} \sigma_i,\, ~\mathbf{j} \in\{\mathbf{v}, \mathbf{w} \}, 
 \end{align*} 
 where $\mathcal{I}_\mathbf{v} = \mathbb{I}_{[1,q]}$ and $\mathcal{I}_\mathbf{w} = \mathbb{I}_{[q+1, Nm]}$.\End
\end{lemma}
\begin{proof} 
The linearity of the trace operation gives
\begin{align*}
    &\ \frac{1}{|\Xini|}\int_{\Xini}\nabla_\mathbf{v} J(x,\mathbf{u}^\star(x))^\top\nabla_\mathbf{v} J(x,\mathbf{u}^\star(x))\dx  \\
&=\frac{1}{|\Xini|}\int_{\Xini}\text{trace}\big(\nabla_\mathbf{v} J(\cdot)\nabla_\mathbf{v} J(\cdot)^\top\big)\dx \\
            &=\text{trace}\left(T_1^\top\left(\frac{1}{|\Xini|}\int_{\Xini}\nabla_\mathbf{u} J(\cdot)\nabla_\mathbf{u} J(\cdot)^\top\dx\right)T_1\right)      \end{align*}
The last expression is equivalent to $\text{trace}(T_1^\top C T_1) \, = \sum_{i=1}^q \sigma_i$, hence we obtain the assertion for $\mathbf{j} = \mathbf{v}$. The case  $\mathbf{j} = \mathbf{w}$ follows similarly.
\PfEnd
\end{proof}

To the end of computing an approximation of the matrix $C$, consider a randomly chosen set of samples 
\[
\Xtrain \doteq \left\{x^j \in \Xini, j =1,...,N_{\text{train}}\right\}.
\]
Then a numerically tractable approximation of $C$ from \eqref{eq:pd} is given by 
\begin{equation} \label{eq:Chat}
\widehat C = \frac{1}{N_{\text{train}}} \sum_{j=1}^{N_{\text{train}}} S(x^{j})S(x^j)^\top.
\end{equation}
Algorithm~\ref{alg:T1construction} summarizes the computation of $T_1$ and $T_2$ for a given sensitivity function $S: \Xini \to \mathbb{R}^{Nm}$. It also shows how for given $q\in \mathbb{N}$  the active subspace $\mathcal{AS}$ is obtained.

\begin{algorithm}[t]
	\caption{Computation of projectors $T_1$ and $T_2$ }\label{alg:T1construction}
	\algorithmicrequire  $q\in \mathbb{I}_{[1, Nm]}$, $\Xtrain$, $S: \Xini \to \mathbb{R}^{Nm}$
	~\quad \\\algorithmicensure $T_1$, $T_2$
	\begin{algorithmic}[1]
		\State For $j =1,...,N_{\text{train}}$ get $x_j \in \Xtrain$
		\State Solve $\mathscr{P}(x^j)$ from \eqref{eq:fullOCP} to evaluate $S(x^j)$, cf. Remark~\ref{rem:choiceS}
		\State Construct $\widehat C $ via \eqref{eq:Chat} 
		\State Obtain $T$ by eigendecomposition of $\widehat C$, cf. \eqref{eq:eigen_decomposition}
		\State Return $T_1$ as the first $q$ columns of $T$ corresponding to the largest $q$ eigenvalues of $\widehat C$ and $T_2$ as the other $Nm-q$ columns.
	\end{algorithmic}
\end{algorithm}

\begin{remark}[Choice of $S(x)$] \label{rem:choiceS}
Observe that  recursive feasibility of the proposed active-subspace NMPC does not depend on the actual choice of the subspace. Hence, the design of $S(x)$---which in turn determines $T$---is a degree of freedom. 
For example, with $S(x)=u^\star(x)$, our generic active-subspace framework includes SVD with respect to $u^\star(x)$ as suggested by \cite{Bemporad20,Unger12}. Moreover, with $S(x)= I$ and $q< N$ (supposing $\dim u = m =1$), our approach is closely linked---yet not identical---to a move blocking structure, i.e., the first few control steps are considered most important.

Moreover, we remark that for unconstrained problems and for cases in which  no constraint is active, optimality implies $\nabla_{\mathbf{u}}J(x,\mathbf{u}^\star)=0$. Hence, if only a few sampled OCPs admit active constraints, the choice $S(x) = \nabla_\mathbf{u} J(x,\mathbf{u}^\star(x))$ does likely not provide sufficient information. In this case one may include the subspace spanned by corresponding local LQR solutions obtained for the linearised system at the equilibrium in the active subspace $\mathcal{AS}$.  \End
\end{remark}

\section{Simulation Results}\label{Sec:results}
We apply the proposed active-subspace approach to three examples: the longitudinal beam dynamics of a synchrontron~\cite{Faulwasser14p} (i.e. linear stabilizing MPC), a robot with 2 Degrees of Freedom (2-DoF)~\cite{Siciliano10} (i.e.  stabilizing NMPC), and a chemical reactor (i.e. economic NMPC)~\cite{epfl:faulwasser15g}.

In all numerical results reported below, the initial feasible control sequences $\tilde{\mathbf{u}}_0$ in Algorithm~\ref{alg:reducedMPC} are obtained solving a variant of \eqref{eq:fullOCP} with a constant objective function instead of \eqref{eq:condensedObj}. 
Moreover, we illustrate the efficacy of the proposed scheme by comparing memory footprint and computation time among the full-order OCP~\eqref{eq:OCPk}, its condensed form~\eqref{eq:fullOCP}, and the reduced-order OCP \eqref{eq:PR}. To foster comparability, these OCPs are all initialized  with the same feasible control sequence $\tilde{\mbf{u}}=T_1\tilde{\mbf{v}}+T_2\tilde{\mbf{w}}$, i.e., we initialize \eqref{eq:PR} with $\mbf{v}=\tilde{\mbf{v}}$ and $\mu=1$ accordingly.
The computations are done on a virtual machine with 8 virtual CPUs with 2.1 GHz, 8 GB of RAM, and MATLAB R2019a. The linear-quadratic example uses  CasADi (v3.5.5)~\cite{Andersson19} and the QP solver OSQP~\cite{osqp}; while IPOPT~\cite{ipopt}  is considered for the non-convex problems.
\subsection{Linear Synchrotron Example}
\begin{table*}[t]
	\begin{center}
		\caption{Synchrotron example: relative performance difference $\Delta$ \eqref{eq:performanceLoss}  with $N_\text{train}=120$, $N_{x0}=100$, and $Nm=120$.  $N_\text{fea}$ denotes the number of feasible initial conditions, $\bar N$ is the  shortened horizon, $M$ stands for mean and $SD$ for standard deviation. }\label{tab:PerformanceLinear}
		\begin{tabular}{ccccccc||cccc}	
				\toprule
				\multirow{2}{*}{\# dec.}	&\multicolumn{2}{c}{$S= I $} & \multicolumn{2}{c}{ $S= \mathbf{u}^\star $}   &  \multicolumn{2}{c||}{  $S=   \nabla_{\mathbf{u}}J(x,\mathbf{u}^\star)  $} & \multicolumn{4}{c}{ Short-horizon MPC} \\
				&  M (\%)	& SD (\%)& M (\%)& SD (\%) & M (\%)& SD (\%) &$\bar{N}$ & $\frac{N_\text{fea}}{N_\text{x0}}$ (\%)  &M (\%) & SD (\%)	\\
				\midrule
				6 & 2.7$\times 10^{-2}$& 0.14&	2.3$\times 10^{-2}$ &	4.7$\times 10^{-2}$& 2.3$\times 10^{-2}$  &	 4.6$\times 10^{-3}$ &3&  63 &3.7$\times 10^{-3}$ &	5.6$\times 10^{-2}$\\
				12 & 1.4$\times 10^{-2}$&	9.4$\times 10^{-2}$ &	 6.5$\times 10^{-3}$&	1.6$\times 10^{-2}$ & 6.7$\times 10^{-3}$&	1.4$\times 10^{-2}$   & 6& 79 &9.4$\times 10^{-4}$ &	5.5$\times 10^{-5}$\\
				24 & 	3.4$\times 10^{-3}$& 2.4$\times 10^{-2}$ &		8.0$\times 10^{-3}$& 1.8$\times 10^{-2}$ &	4.8$\times 10^{-4}$& 4.4$\times 10^{-3}$ 
				&12&92& 8.5$\times 10^{-3}$ &	2.6$\times 10^{-2}$ 
				\\
				\bottomrule
			\end{tabular}
		\end{center}
	\end{table*}
We consider a linear system $x^+= Ax +Bu$ with system matrices
\[
A = \left[\begin{smallmatrix} 
	1.01 & -0.039 & 0 & 0 \\
	-0.29 & 1.01 & 0 & 0 \\
	0 & 0 & 1.06 & -0.080 \\
	0 & 0 & -1.64 & 1.06\\
\end{smallmatrix}\right],~
B = \left[\begin{smallmatrix} -0.0056 & 0 \\ 0.29 &0  \\ 0 & 0.0023 \\ 0& -0.058 \end{smallmatrix}\right],
\]
which is obtained from the exact discretization with sampling interval $h = \SI{2.66}{\mu s}$ of the continuous-time longitudinal beam dynamics of a synchrotron given in~\cite{Faulwasser14p}. The inputs and states are restricted by $\mbb{U}= [-0.5,\ 0.5] \times [-0.1,\ 0.1] $, $\mbb{X} = [-0.32,\ 0.32] \times [-0.4,\ 0.4]\times  [-0.1,\ 0.1] \times  [-0.46,\ 0.46].$

The aim is to stabilize the system at the origin with the quadratic stage cost $\ell(x,u) = \frac{1}{2} (\|x\|_Q^2+\|u\|_R^2)$, $Q=\diag(2, 8, 2,8)$, $R = \diag(0.9, 0.9)$, and the terminal cost $\Vf(x_{k+N|k})=\|x_{k+N|k}\|_P^2$ with $P$ as the unique solution to the algebraic Riccati equation. With horizon $N=60$, there are $Nm=120$ input decision variables for the full-order problem~\eqref{eq:OCPk}.  The terminal region $ \mbb{X}_{\text{f}}=\{x\in \X \, | \,  A_{\text{f}}x\leq b_\text{f}\}$ is constructed to be the maximal positive invariant set for the autonomous system $x_{k+1}= (A+B K_{\text{LQR}})x_k$ with LQR feedback gain $K_{\text{LQR}}$, $x_k \in \X$, and $K_{\text{LQR}} x_k\in \mbb{U}$, cf. command LQRSet() in MPT3 Toolbox \cite{MPT3}.

After reformulating \eqref{eq:OCPk} to \eqref{eq:fullOCP}, it is well known that the condensed problem~\eqref{eq:fullOCP} is ill-conditioned for unstable linear systems \cite{Maciejowski02}. Thus, we prestabilize the linear system by setting
\[
u_{k+i|k} = K_{\text{LQR}}( x_{k+i|k} - \bar{x}) + u^c_{k+i|k}  
\]
and employ $\mathbf{u}^c_k \doteq [u_{k|k}^{c\top}, ..., u_{k+N-1|k}^{c\top}]^\top$ as the new decision variables of \eqref{eq:fullOCP}. Moreover, the active subspace method is exploited to reduce the dimensionality of $\mathbf{u}^c$, cf. Sections \ref{Sec:activesubspace} and \ref{sec:sensi}. 
With $N_{\text{train}}=120$ samples in \eqref{eq:Chat}, we learn active subspaces for different $S(x)$, cf. Algorithm~\ref{alg:T1construction} and Remark~\ref{rem:choiceS}.

To compare the performances of active-subspace MPC and full-order MPC, we sample $N_{x0}=100$ random initial conditions $ x^j \in \Xini,\,j\in \mathbb{I}_{[1,N_{x0}]}$. For these samples we evaluate the relative closed-loop performance difference
\begin{equation}\label{eq:performanceLoss}
	\textstyle \Delta =  \frac{1}{N_{x0}} \sum_{j=1}^{N_{x0}} \left(\widehat{J}_\text{c}(x^j)-J_\text{c}\left(x^j\right)\right)\cdot \left(J_\text{c}(x^j)\right)^{-1},
\end{equation}
where $\widehat{J}_\text{c}$ denotes the closed-loop performance for active-subspace MPC and $J_\text{c}$ is the closed-loop performance for full-order MPC. Furthermore, to illustrate the advantages of using active subspace rather than simply directly reducing the prediction horizon, we consider the closed-loop performance losses of short-horizon MPC problems in comparison to the original full-order problem with long horizon. The short-horizon MPCs are constructed by shrinking the horizon of the original full-order MPC such that the same number of decision variables as in the reduced-order problems is attained.

	\begin{figure}[t]
		\includegraphics[width=0.85\columnwidth]{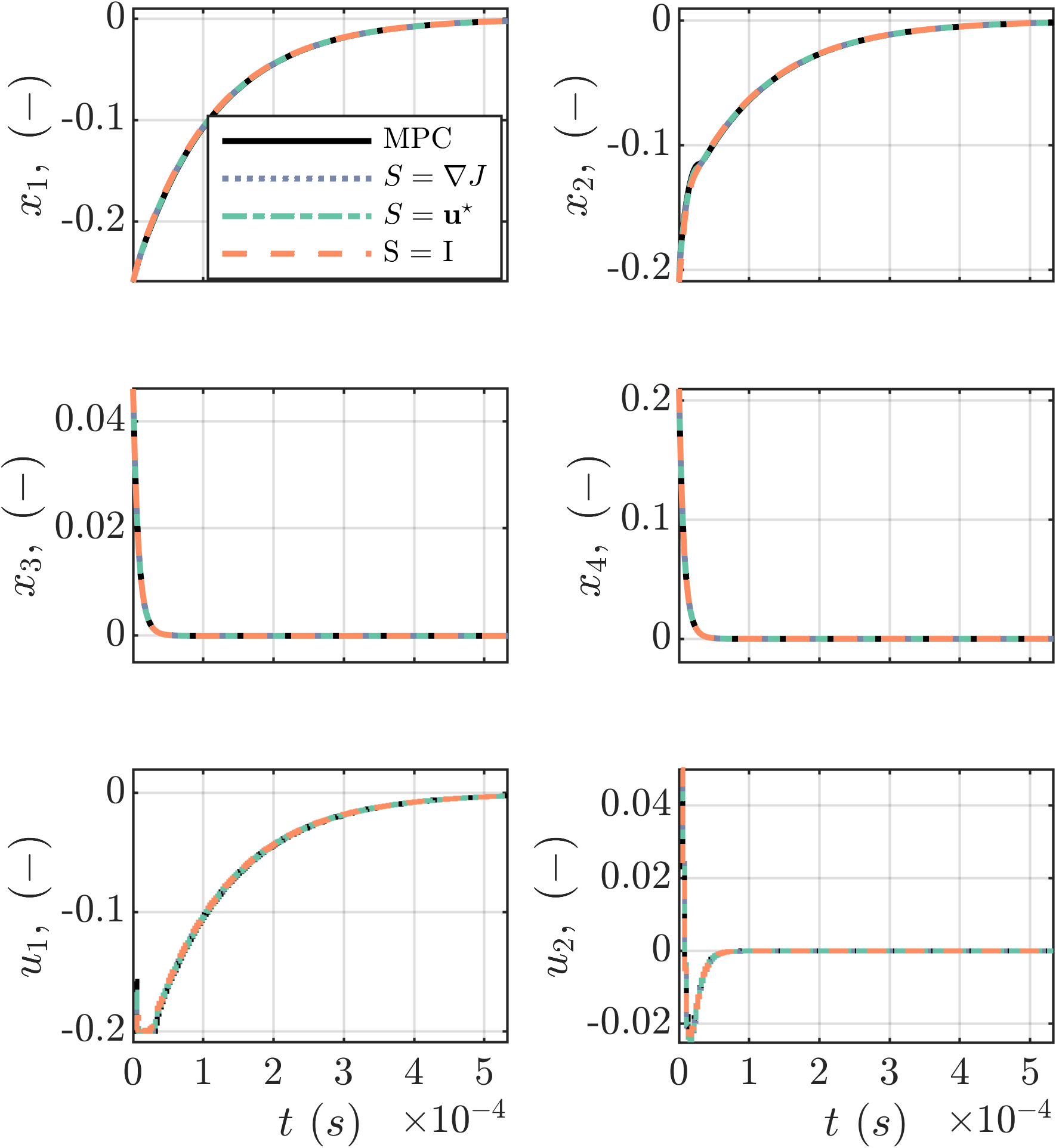}
		\caption{Synchrotron example: closed-loop trajectories for different choices of $S(x)$ with \# dec. = 6, and $x_0 = [0.26, -0.21,$ $0.046,0.21]^\top $. }\label{fig:linearTrajectories}
	\end{figure}

Table \ref{tab:PerformanceLinear} depicts the averaged closed-loop performance loss  for different choices of $S(x)$ and the corresponding short-horizon MPCs. We observe that even with a drastically reduced number of decision variables, i.e. $6$ compared to $Nm=120$, the performance loss of the proposed scheme is minuscule  ($10^{-2} \%$). As depicted in Figure~\ref{fig:linearTrajectories}, with $6$ decision variables the closed-loop trajectories of the reduced-order problems with different choices of $S(x)$ show no obvious deviation from the full-order MPC trajectory.  Note that with the chosen terminal cost, i.e. $P$ as the unique solution of the algebraic Riccati equation, the solutions of MPC with any prediction horizon align with the infinite-horizon MPC if the terminal constraints are not active. Therefore, we see the performance losses of short-horizon MPCs are also relatively small, which indicates that the terminal constraints are not frequently active in this case. However, the short-horizon MPCs encounter infeasibility for some initial conditions. The percentage of feasible OCPs for the short horizon MPC is denoted in the column $\frac{N_\text{fea}}{N_\text{x0}}$. We remark that the active subspaces schemes inherit the feasibility of the original full-order MPC, i.e., for them we have that $\frac{N_\text{fea}}{N_\text{x0}}[\%] = 100[\%]$, cf. Proposition~\ref{prop:recursive_feasibility}.

	\begin{table}[t]
	\begin{center}
		\caption{Synchrotron example: memory footprint and relative computation time of different OCPs with $100$ initial conditions and $10$ runs per initial condition, $S=  \nabla_{\mathbf{u}}J(x,\mathbf{u}^\star)$, solver: OSQP}\label{tab:CompTimeLinear}
		\begin{tabular}{cccccc} 	
			\toprule
			\multirow{2}{*}{OCP}	 & \multirow{2}{*}{ \# dec.} & \multirow{2}{*}{\# Jac.} & \multirow{2}{*}{\# Hes.}  &     \multicolumn{2}{c}{ $T_{com}/ T_{(5)}$  }  \\
			& & & & M (\%) & SD (\%)	\\
			\midrule
			\eqref{eq:fullOCP}	&120 & 24000 &	7260		& 100 &0 \\
			\eqref{eq:OCPk} &364 &2612 &2176		& 20.1 &	4.30\\
			\eqref{eq:PR}&	6 &2292&36&	12.2&	3.58\\
			\eqref{eq:PR}&12& 4584	&144&	15.3&	4.01\\ 
			\eqref{eq:PR}&24 &9168	&576		&	24.7	&5.86	\\
			\bottomrule
		\end{tabular}
	\end{center}
\end{table}

\begin{table*}[t]
	\begin{center}
		\caption{2-DoF robot: relative performance difference $\Delta$ \eqref{eq:performanceLoss} with $N_{\text{train}}=120$, $N_{x0}=20$, and $Nm=120$.    $N_\text{fea}$ denotes the number of feasible initial conditions, $\bar N$ is the  shortened horizon, $M$ stands for mean and $SD$ for standard deviation. }\label{tab:PerformanceRobot}
		\begin{tabular}{c cccccc|| c c c c}	
			\toprule
			\multirow{2}{*}{ \# dec.} 	&\multicolumn{2}{c}{$S= I $} & \multicolumn{2}{c}{ $S= \mathbf{u}^\star $}   &  \multicolumn{2}{c||}{ $S=   \nabla_{\mathbf{u}}J(x,\mathbf{u}^\star)  $ }& \multicolumn{4}{c}{ Short-horizon NMPC}   \\
			&  M (\%)	& SD (\%) & M (\%)& S (\%) & M (\%) & SD (\%) & $\bar{N}$ &$\frac{N_\text{fea}}{N_\text{x0}}$ (\%) & M (\%) & SD (\%)	\\
			\midrule
			6   &14.1 & 9.52 & 11.9 & 10.8 &6.23 & 8.62& 3&0& \multicolumn{2}{c}{no feasible}\\
			12  &9.30& 8.10 & 2.67 & 4.54 &2.58 & 6.17 &6& 0 &\multicolumn{2}{c}{solutions}\\ 
			24 & 5.60& 6.75 &0.53 &1.34	& 1.28& 4.32 & 12& 0& \multicolumn{2}{c}{found}\\
			\bottomrule
		\end{tabular}
	\end{center}
\end{table*}

Table~\ref{tab:CompTimeLinear} lists the memory footprint and the computation time for evaluating different OCPs, i.e. the full-order OCP~\eqref{eq:OCPk}, the condensed OCP~\eqref{eq:fullOCP}, and the reduced-order OCP~\eqref{eq:PR}. Here \# dec. is the number of decision variables, \# Jac. is the number of nonzeros in constraint Jacobians, and \# Hes. is the number of nonzeros in the Hessian of the objective. Moreover, $T_{com}/T_{(5)}$ denotes the relative computation time of different OCPs compared to OCP~\eqref{eq:fullOCP} for each initial condition.
 Due to the sparse structure of OCP~\eqref{eq:OCPk}, although with more decision variables, its Jacobian and Hessian contain less non-zero elements than the condensed OCP~\eqref{eq:fullOCP}.
Since the QP solver OSQP exploits sparsity, we observe a reduction of computation time about $80\%$ for the non-condensed OCP~\eqref{eq:OCPk} compared to OCP~\eqref{eq:fullOCP}.
Though the reduced-order OCP~\eqref{eq:PR} loses  sparsity, the memory footprint with small (reduced) dimension is comparable to \eqref{eq:OCPk}. Moreover, with $6$ decision variables, the computation time of the reduced-order OCP is about $60\%$ of the non-condensed full-order OCP~\eqref{eq:OCPk}.

\subsection{Nonlinear 2-DoF Robot Example}

The dynamics of a 2-DoF robot are 
\begin{equation}\label{eq:2dofrobot}
B(\theta)\ddot{\theta} +C(\theta,\dot{\theta})\dot{\theta}+g(\theta) = u,
\end{equation}
with joint angles $\theta=[\theta_1, \theta_2]^\top \in \mathbb{R}^2$ in [rad] and input torques $u=[u_1,u_2]^\top \in \mathbb{R}^2$ in [$\SI{}{Nm}$]. The details of the coefficient matrices  and the system parameters are given in \cite{Siciliano10}. We collect $(\theta_1,\theta_2,\dot\theta_1,\dot\theta_2)^\top$ in the state vector $x$. States and inputs are subject to the constraints $u\in [-1000, 1000]$ $\SI{}{Nm}$ and $\dot \theta \in [-\frac{3}{2} \pi, \frac{3}{2} \pi]$ $\SI{}{rad/s}$. We obtain a discrete time system by considering the fixed step size $\SI{0.05}{s}$ and a 4th-order Runge-Kutta scheme.

The control task is to stabilize the system at $\bar{x} = [\frac{\pi}{2}, 0, 0, 0]^\top$ and $\bar{u} = [ 0, 0]^\top$; thus we specify the stage cost as $\ell(x,u) = \frac{1}{2} (\|x-\bar{x}\|_Q^2+\|u-\bar{u}\|_R^2)$ with $Q=\diag(5, 1, 0.5,0.5)$, $R = \diag(1, 1)$. With the prediction horizon $N=60$, there are $Nm=120$ input decision variables for the full-order problem~\eqref{eq:OCPk}.  Moreover, we employ a terminal equality constraint $x_{k+N|k}=\bar x$ to ensure feasibility and stability.

\begin{figure}[t]
	\includegraphics[width=0.85\columnwidth]{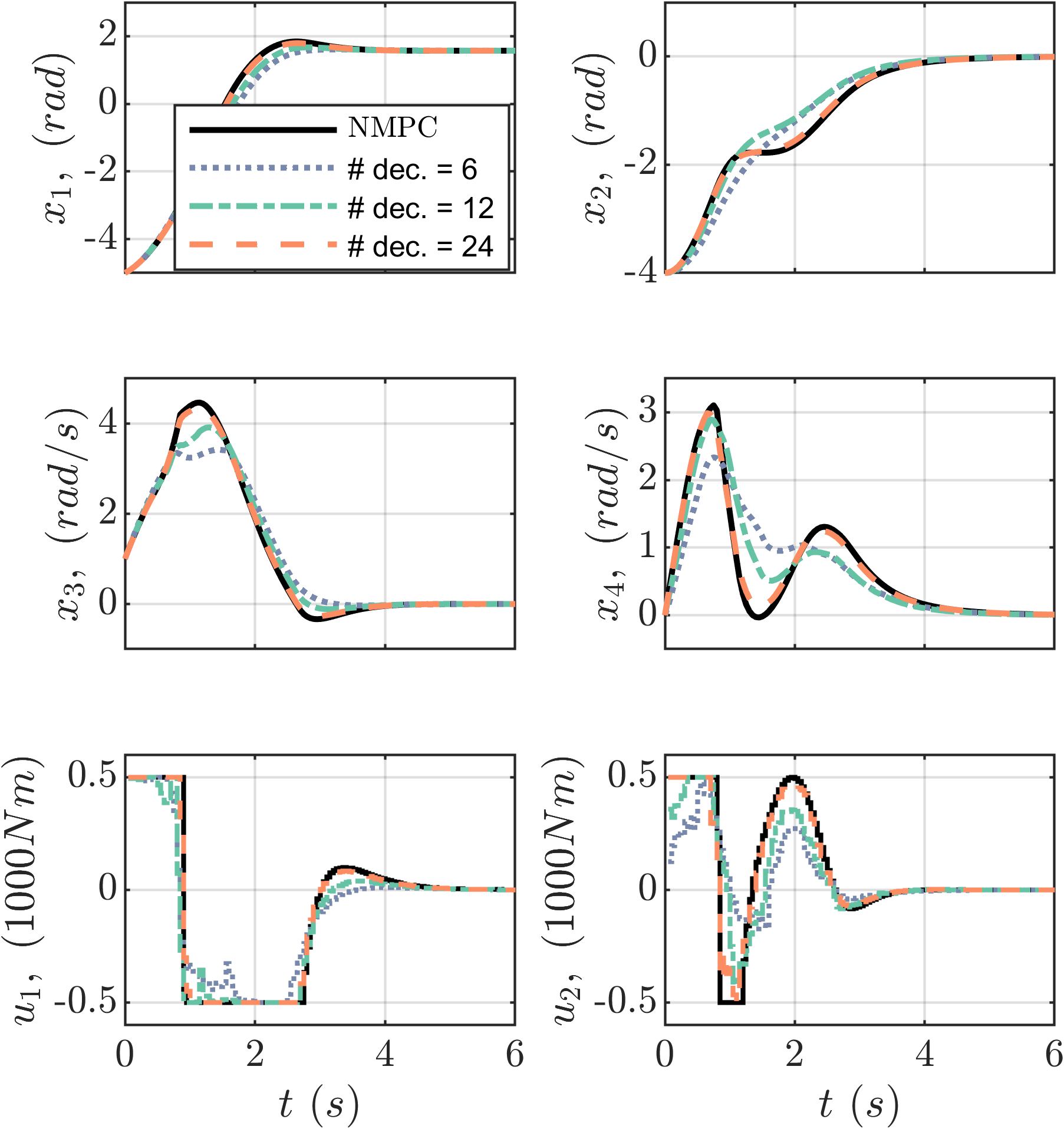}
	\caption{2-DoF robot: closed-loop robot trajectories for $S=\nabla_\mathbf{u}J(x,\mathbf{u}^\star)$ with $q+1=6,12,24$ and $x_0 = [-5, -4, $ $1, 0]^\top $. }\label{fig:RobotTrajectories}
\end{figure}  

\begin{table}[t]
	\begin{center}
		\caption{2-DoF robot: memory footprint and relative computation time of different OCPs with $N_{x0}=20$, $Nm=120$, and $S=  \nabla_{\mathbf{u}}J(x,\mathbf{u}^\star)$, solver IPOPT.  }\label{tab:CompTimeRobot}
		\begin{tabular}{cccccc}	
			\toprule
			\multirow{2}{*}{OCP} &\multirow{2}{*}{ \# dec.} 	 & \multirow{2}{*}{\# Jac.} & \multirow{2}{*}{\# Hes.}   &    \multicolumn{2}{c}{ $T_{com}/ T_{(5)}$  }  \\
			& & &  & M (\%) & SD (\%)	\\
			\midrule
			\eqref{eq:fullOCP} &	120&7920&	7260& 100 &0 \\
			\eqref{eq:OCPk} &	364&1928 &	1270 &6.01 &	1.04
			\\
			\eqref{eq:PR} &6	& 1464 &	21	 & 12.8&	3.06
			\\
			\eqref{eq:PR}&12  &2928	&78		 &20.2&	5.79
			\\ 
			\eqref{eq:PR}&24 &5856&	300	 & 35.2 &	7.87
			\\
			\bottomrule
		\end{tabular}
	\end{center}
\end{table}
Figure \ref{fig:RobotTrajectories} shows closed-loop trajectories for the initial condition $x_0 = [-5, -4, 1, 0]^\top$---for different numbers of decision variables \# dec.$=6,12,24$ obtained for $S(x) = \nabla_{\mathbf{u}}J(x,\mathbf{u}^\star)$---and the solution of the full-order NMPC. As one can see, all trajectories converge to the steady state after $t=\SI{4}{s}$. Specifically, the trajectories for $q+1=24$ show no obvious deviation from the full-order NMPC trajectories.

With $N_{\text{train}}=120$ and $N_{x0}=20$, we summarize the  performance differences for different choices of $S(x)$ in Table~\ref{tab:PerformanceRobot}. With the proposed active-subspace NMPC scheme ($S= \nabla_{\mathbf{u}}J(x,\mathbf{u}^\star)$), one  achieves the performance loss of $2.6\%$ with $12$ decision variables, i.e.,  dimensionality reduction by factor $10$.
Moreover, with $24$ decision variables, the approach with $S= \mathbf{u}^\star$ reaches around $0.5\%$ performance loss compared to $5\%$ performance loss using $S= I$. Observe that no feasible solution is found for the corresponding short-horizon NMPCs with the same number of decision variables.

Moreover, we compare the memory footprint and the relative computation time of the different OCPs in Table~\ref{tab:CompTimeRobot}. Since the solver IPOPT does utilize sparsity to accelerate computation, the non-condensed OCP~\eqref{eq:OCPk} is solved fastest. Although \eqref{eq:PR} is condensed,  we remark that the proposed scheme with $12$ decision variables achieves comparable computation time and less memory footprint than OCP~\eqref{eq:OCPk}. As indicated before, the closed-loop performance loss in this case is around $2.6\%$.

\subsection{Nonlinear Van de Vusse Reactor}

\begin{table*}[t]
	\begin{center}
		\caption{ CSTR example: relative performance difference $\Delta$ \eqref{eq:performanceLoss} with $N_{x0}=20$ and $Nm=200$.  $N_\text{fea}$ denotes the number of feasible initial conditions,  $\bar N$ is the  shortened horizon, $M$ stands for mean and $SD$ for standard deviation.  }\label{tab:PerformanceCSTR}
		\begin{tabular}{c cccccc||c c  cc}	
				\toprule
				\multirow{2}{*}{\# dec.}	&\multicolumn{2}{c}{$S= I $} & \multicolumn{2}{c}{ $S= \mathbf{u}^\star $}   &  \multicolumn{2}{c||}{ $S=   \nabla_{\mathbf{u}}J(x,\mathbf{u}^\star)  $} & \multicolumn{4}{c}{ Short-horizon NMPC}   \\
				&  M (\%)	& SD (\%)& M (\%)& SD (\%) & M (\%) & SD (\%) & $\bar{N}$ &$\frac{N_\text{fea}}{N_\text{x0}}$ (\%) & M (\%) & SD (\%)	\\
				\midrule
				5 & 2.2$\times 10^{-2}$	& 2.7$\times 10^{-2}$	&	3.8$\times 10^{-2}$	 &	2.4$\times 10^{-2}$ & 0.38&	0.37& 3 & 100& 20.7 &	2.12\\
				10  &2.9$\times 10^{-3}$& 4.8$\times 10^{-3}$ &2.6$\times 10^{-3}$&	6.1$\times 10^{-3}$& 1.5$\times 10^{-2}$&	1.2$\times 10^{-2}$ &5& 100& 10.9 &	1.04
				\\ 
				20& 5.1$\times 10^{-4}$&	9.2$\times 10^{-4}$&	2.5$\times 10^{-4}$&	8.0$\times 10^{-4}$&	8.3$\times 10^{-4}$&9.6$\times 10^{-4}$ & 10&100& 1.06&	8.6$\times 10^{-2}$
				\\
				\bottomrule
			\end{tabular}
		\end{center}
	\end{table*}

	\begin{table}[t]
		\begin{center}
			\caption{CSTR example: memory footprint and relative computation time of different OCPs with $N_{x0}=20$, $Nm=200$, and $S=  \nabla_{\mathbf{u}}J(x,\mathbf{u}^\star)$), solver IPOPT.}\label{tab:CompTimeCSTR}
			\begin{tabular}{cccccc}	
				\toprule
				\multirow{2}{*}{OCP}& \multirow{2}{*}{ \# dec.} 	 & \multirow{2}{*}{\# Jac.} & \multirow{2}{*}{\# Hes.}     & \multicolumn{2}{c}{ $T_{com}/ T_{(5)}$  }  \\
				& & & & M (\%) & SD (\%)\\
				\midrule
				\eqref{eq:fullOCP}&200  & 30500&	20100		& 100 &0 \\
				\eqref{eq:OCPk}&	604	 &3004 &	1500  &	2.99 & 0.33\\
				\eqref{eq:PR} &	5 &2500 &	15 &	2.81&	0.45	
				\\
				\eqref{eq:PR}&	10 &5000	&55&	5.56&	0.90
				\\ 
				\eqref{eq:PR}&	20 &10000 &	210&	11.7	&1.38	\\
				\bottomrule
			\end{tabular}
		\end{center}
	\end{table}

We consider \edit{the partial model of} Van de Vusse Continuous Stirred Tank Reactor (CSTR) in which the reactions $A \xrightarrow{k_1} B\xrightarrow{k_2} C $, $2A \xrightarrow{k_3} D$ take place \cite{Rothfuss96}.
The concentrations $c_A, c_B$ [$\SI{}{mol/m^3}$] and  the reactor temperature $ T$  [$\SI{}{^\circ C}$] satisfy the dynamics
\begin{subequations} \label{eq:sys_vanVuss}
	\begin{align}
	\dot c_A & = r_A + (c_{\text{In}}-c_A)u_1 , \quad \dot c_B  = r_B  -c_Bu_1, \\
	\dot  T  & = h + \alpha(u_2- T) + ( T_{\text{In}}- T)u_1, 
	\end{align}
where 
\begin{align*}
	r_A	   &=-k_1c_A -2k_3c_A^2, 	\quad r_B= k_1 c_A- k_2 c_B ,\\ 
	 h &= -\delta\Big( k_1 c_A \Delta H_{AB} + k_2 c_B \Delta H_{BC}  + 2k_3c_A^2 \Delta H_{AD} \Big), 	
\end{align*}
\end{subequations}
and $k_i = k_{i0}\exp{\dfrac{-E_i}{ T+ T_0}}, \, i = 1,2,3$.
The inputs $u_1, u_2$ are the normalized flow rate through the
reactor [$\SI{}{1/h}$] and the temperature in the cooling jacket [$\SI{}{^\circ C}$], respectively. The system parameters can be found in \cite{Rothfuss96}.  We consider the states $x= [c_A,c_B,T]^\top$ and the inputs $u =[u_1,u_2]^\top$ with constraints
$c_A \in [0, 6000] \frac{\text{mol}}{\text{m}^3}$,  $c_B \in [0, 4000]\frac{\text{mol}}{\text{m}^3}$, ~  $T\in [70, 200]^\circ \text{C}$,
 $u_1 \in [3, 35]\frac{1}{\text{h}}$,   ~$u_2 \in [100, 200]^\circ \text{C}$. In comparison to \cite{Rothfuss96} and similar to \cite{epfl:faulwasser15g} we omit the dynamics of the cooling jacket.   
Moreover, we rescale inputs and states to avoid numerical difficulties; precisely, $c_A$ and $c_B$ are scaled by $10^{-3}$, $u_2$ and $T$ by $10^{-2}$, and $u_1$ by $10^{-1}$.  We use the sampling time \edit{ $2\times 10^{-3}\SI{}{h}$ }and a 4th-order explicit Runge-Kutta scheme to approximate  \eqref{eq:sys_vanVuss} as a discrete-time system. The control task is to optimize the amount of $B$ produced; hence we choose $\ell(x,u)= -c_Bu_1 +  \epsilon u^\top u$ where the second term with $\epsilon = 10^{-3}$ is a small Tikhonov regularization to avoid singular arcs in the optimal solution. The optimal steady state pair (in scaled coordinates) corresponding to $\ell$ reads
$
\bar{x} =  [2.1753,1.1051,1.2849]^\top$, $\bar{u} = [3.5,1.4268]^\top$.

Figure \ref{fig:CSTRturnpike} depicts several open-loop solutions to the full-order OCP for different initial conditions and the horizon $N= 100$. As one can see, the optimal solutions exhibit the turnpike phenomenon, see~\cite{tudo:faulwasser22a}.  Specifically, $u_2^\star$ stays close to $\bar{u}_2$ during the middle part of the horizon. We note that since $u_1^\star(k)\equiv \bar{u}_1$, it is not plotted. 
Notice that we neither employ a terminal cost function nor a terminal constraints for this example as both would potentially eliminate the turnpike leaving arc. Here, however, we are interested in illustrating the performance of the proposed scheme despite the leaving arc. 
The construction of (approximately feasible) input guesses now exploits the turnpike phenomenon, i.e., we split $\mathbf{u}_k$ into two pieces and add $\bar{u}$ in the middle of the horizon, cf. Section~\ref{Sec:activesubspace}. Moreover, we leverage the turnpike phenomenon to shift the inputs
\[
u_{k+i|k} = \bar{u} + u^c_{k+i|k},
\]
where $\bar u$ is optimal steady state input from above. 
Specifically, we employ $\mathbf{u}^c_k \doteq [u_{k|k}^{c\top}, ..., u_{k+N-1|k}^{c\top}]^\top$ as the new decision variables of \eqref{eq:fullOCP}. Similar to the linear example, we apply the proposed active subspace method based on OCP~\eqref{eq:fullOCP} with $\mathbf{u}^c_k$. This way, we focus on the active subspaces that deviate from the turnpike trajectory.	
	
	\begin{figure}[t]
		\includegraphics[width=0.85\columnwidth]{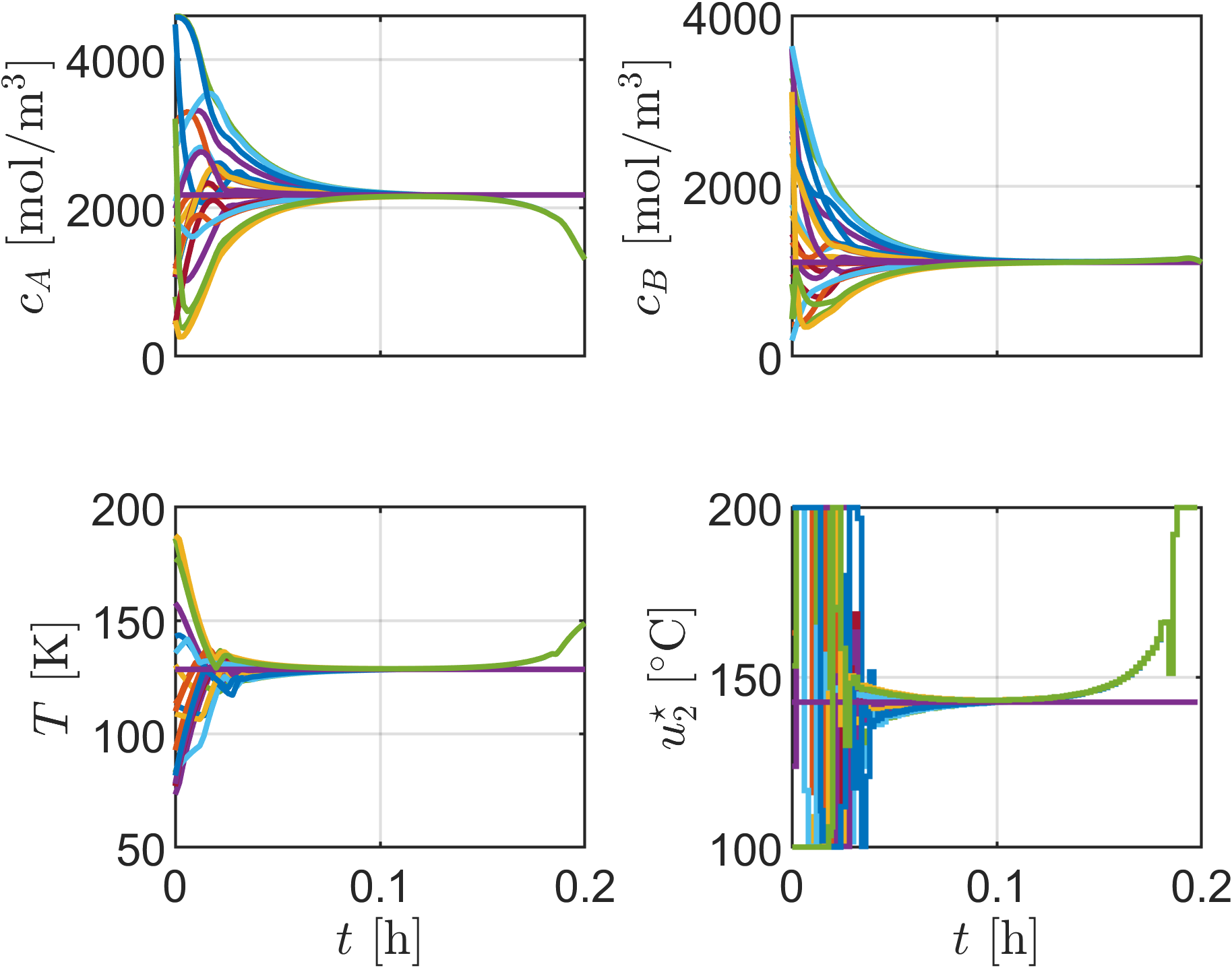}
		\caption{CSTR example: optimal open-loop solutions of the full-order OCP \eqref{eq:OCPk} with different $x_0$. }\label{fig:CSTRturnpike}
	\end{figure}

With $N_{\text{train}}=200$ and $N_{x0}=20$, we summarize the averaged closed-loop performance for different choices of $S(x)$ in Table \ref{tab:PerformanceCSTR}. It shows that one can reduce the dimensionality from $200$ to $5$ with a minuscule performance loss in the range of $10^{-2}\%$ for $S(x) = I$ and $S(x) = \mathbf{u}^\star$, while the performance loss of the corresponding short-horizon NMPC is around $20\%$. Moreover, the closed-loop trajectories depicted in Figure~\ref{fig:CSTRtrajectories} also show a good match among $S(x) = \mathbf{u}^\star$, $S(x) = I$, and the full-order NMPC.  As shown in Table~\ref{tab:CompTimeCSTR}, we achieve reduced memory footprint and similar computation time compared to the non-condensed problem~\eqref{eq:OCPk}.

	\begin{figure}[t]
	\includegraphics[width=0.85\columnwidth]{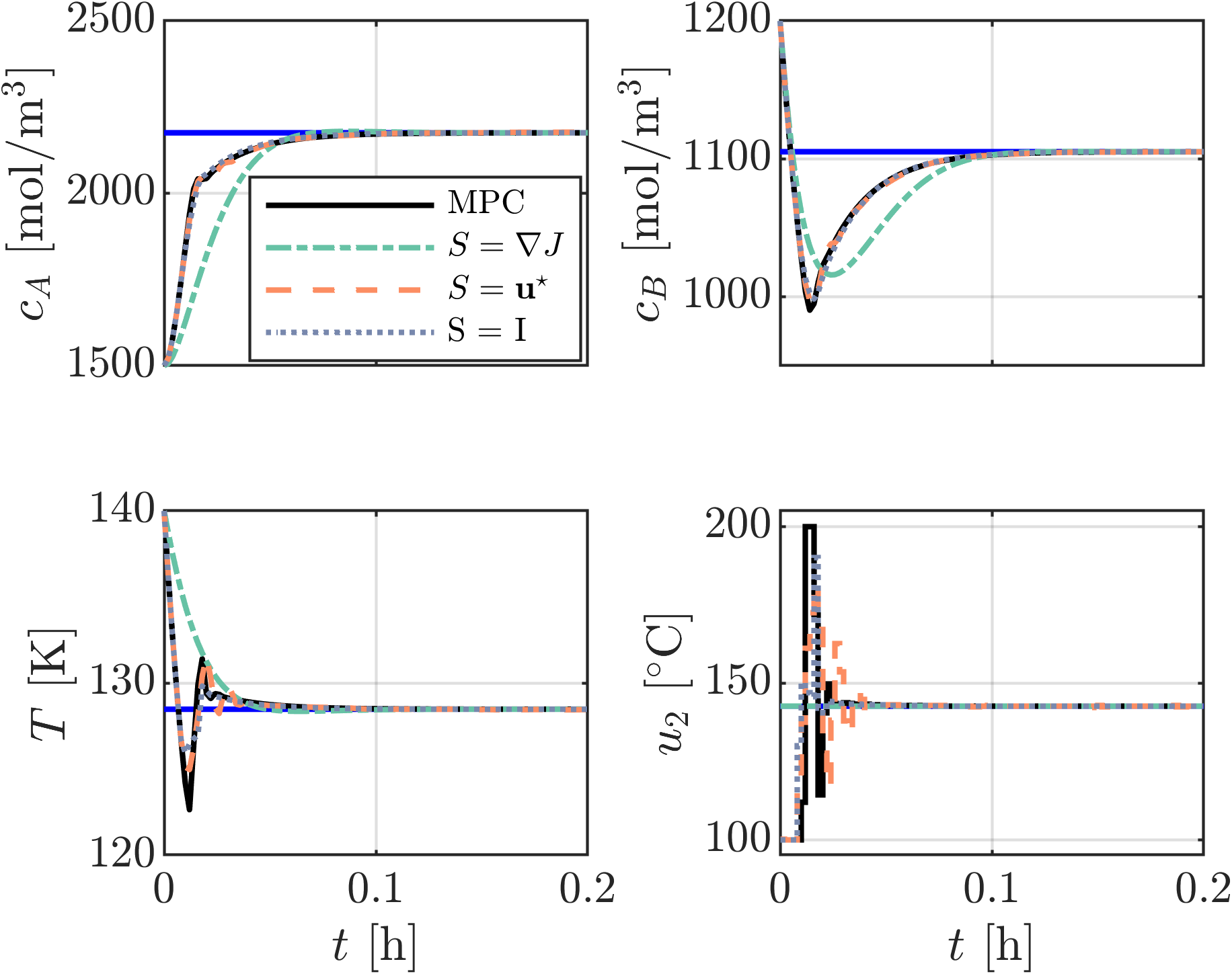}
	\caption{CSTR example: closed-loop trajectories for different choices of $S(x)$ with \# dec. $=5$ and $x_0 = [1500, 1200, 140]^\top $. }\label{fig:CSTRtrajectories}
\end{figure}

\section{Conclusions}\label{sec:con}
This paper has presented a framework to reduce the number of input decision variables in NMPC based on active subspaces. The proposed framework allows the reduction without  sacrificing  closed-loop performance or feasibility. To this end, we have shown that by inheriting the inactive part of a feasible guess, recursive feasibility of the reduced dimensional optimization problem can be enforced. Moreover,  constructing the reduced dimensional active subspace via global sensitivity analysis, linear and nonlinear examples for stabilzing and economic MPC formulations have illustrated significant dimensionality reduction  with only minor loss of closed-loop performance.  In contrast, we observe that nominal NMPC schemes which employ the same number of decision variables (via shorter horizons) as used in the reduced subspaces  either encounter infeasibility (if terminal constraints are present) or they result in significant closed-loop performance loss.
Future work should discuss error and performance bounds and online adaptation of subspaces.

\bibliographystyle{IEEEtran}        
\bibliography{All}           

\end{document}